\documentclass[11pt, oneside]{amsart} 
\usepackage{times}
\usepackage[portrait,margin=1in]{geometry} 
\allowdisplaybreaks
\include{style}
\begin{document}

\title[Applications of the cluster graphing]{Applications of the cluster graphing}
\author[Chu]{Tasmin Chu}
\address{
Tasmin Chu\\
Division of Physics, Mathematics and Astronomy \\
California Institute of Technology \\
Pasadena, CA \\
USA
}
\email{tlchu@caltech.edu}

\begin{abstract}
    In 1999, two papers of Benjamini, Lyons, Peres, and Schramm (\cite{BLPS99inv}, \cite{BLPSpc}) showed that for Bernoulli percolation on unimodular nonamenable quasi-transitive graphs, there are no infinite clusters at criticality.
    Using the theory of countable Borel equivalence relations and Gaboriau's cluster graphing construction, we give a short proof of this result.
    We also point out other uses of the cluster graphing construction in the literature, for instance in showing nonuniqueness at $p_u$ in \cite{GABORIAU20161114}.
    The purpose of this short note is to make folklore proofs known to experts more accessible to the wider community of probabilists and measured group theorists.
\end{abstract}

\maketitle

\section{Applications of the cluster graphing}

\subsection{No percolation at criticality on unimodular nonamenable graphs}

The following theorem first appeared in the literature in \cite{BLPS99inv}, with a simpler proof appearing in \cite{BLPSpc}. 
We give a short proof using the cluster graphing construction of \cite{Gaboriau05}, which is described in the \nameref{sec: Appendix} for the convenience of the reader. Briefly: given an invariant percolation $\mathbf{P}$ on a transitive graph $G$, one can construct a corresponding \textit{cluster graphing}, meaning a Borel graph $\mathcal{G}^{\text{cl}}$ on a standard probability space $(Y, \nu)$. Intuitively, the components of $\mathcal{G}^{\text{cl}}$ look like possible percolation configurations, and the measure $\nu$ is essentially the pushforward of $\mathbf{P}$ under a certain canonical map. The connectedness relation $\mathcal{R}^{\text{cl}}$ of $\mathcal{G}^{\text{cl}}$ is an example of a countable Borel equivalence relation; we call $\mathcal{R}^{\text{cl}}$ a \textit{cluster equivalence relation}.
We suggest the reader consult \cite[Section 5.B.1]{CTT22} for a detailed exposition of the cluster graphing; for a description of the cluster graphing in the quasi-transitive setting, see \cite{bell2025heavyrepulsionclustersbernoulli}. The reader who is not familiar with Bernoulli$(p)$ percolation should consult the preliminaries.

The critical observation is that for Bernoulli$(p)$ percolation on a transitive graph, the cluster equivalence relation at criticality $\mathcal{R}_{p_c}$ is always hyperfinite, meaning it can be written as an increasing union of finite Borel equivalence relations. 
This fact was previously observed by Damien Gaboriau (it is attributed to him in a paper of Lyons, see \cite[page 3]{LYONS_2013}).

\begin{thm}\label{NoInfiniteClustersAtCriticality}
    If $G$ is a nonamenable unimodular quasi-transitive graph, then there are no infinite clusters at $p_c(G)$.
\end{thm}

\begin{proof}
    Let $G[p]$ denote the random graph formed by Bernoulli$(p)$ percolation on $G$. Observe that the cluster equivalence relation $\mathcal{R}_{p_c}$ corresponding to $G[p_c]$ is hyperfinite. Indeed, if $p_n$ is a sequence converging to $p_c$ from below, then $\mathcal{R}_{p_c} = \bigcup_n \mathcal{R}_{p_n}$, where $\mathcal{R}_{p_n}$ are (finite) cluster equivalence relations corresponding to $G[p_n]$.
    However, by \cite[Theorem 1.1]{BLS-perturb}, if $G$ is a nonamenable quasi-transitive unimodular graph and $G[p]$ has infinite clusters, then almost surely $G[p]$ contains an invariant random subgraph whose components have positive Cheeger constant. By a theorem of Kaimanovich \cite{kaimanovich1997amenability} (see \cite[Theorem 3.10]{TasminThesis} for an explicit translation to the cluster graphing), such a subgraph implies that $\mathcal{R}_{p}$ is not hyperfinite. Thus, if $G[p]$ contains infinite clusters, the corresponding cluster equivalence relation $\mathcal{R}_p$ is not hyperfinite. We conclude that $G[p_c]$ cannot have infinite clusters, i.e. there is no percolation at criticality. 
\end{proof}

In personal communication, Robin Tucker-Drob pointed out the following modification of the proof above, which bypasses \cite[Theorem 1.1]{BLS-perturb} entirely and only uses elementary considerations about countable Borel equivalence relations.

\begin{proof}
    Let $G[p]$ denote the random graph formed by Bernoulli$(p)$ percolation on $G$, with corresponding law $\pr_p$. As above, the cluster equivalence relation $\mathcal{R}_{p_c}$ corresponding to $G[p_c]$ is hyperfinite. By a theorem of Newman and Schulman \cite{newman1981infinite}, the number $N_\infty$ of infinite clusters at any level $p$ is a $\pr_p$-a.s. constant, taking values in $\{0,1, \infty\}$.
    
    We first rule out a unique infinite cluster at $p_c$. Indeed, the nonamenability of $G$ gives that the ambient equivalence relation $\mathcal{R}^{\text{full}}$ (called the full equivalence relation in \cite{Gaboriau05}) is not hyperfinite (this can be deduced from \cite{kaimanovich1997amenability}). 
    The unique infinite cluster in almost every $\mathcal{R}^{\text{full}}$-class gives a complete section $B$ of $\mathcal{R}^{\text{full}}$, so that $\mathcal{R}^{\text{full}}|_B = \mathcal{R}_{p_c}|_B$. Since $\mathcal{R}^{\text{full}}$ is not hyperfinite, $\mathcal{R}_{p_c}$ is also not hyperfinite by \cite[Proposition 1.3 vi)]{JKLcber}, a contradiction.

    We now rule out the possibility of infinitely many infinite clusters at $p_c$. Indeed, using insertion tolerance, by opening finitely many edges, one can connect the origin to at least three distinct infinite clusters. 
    Thus, with positive probability, the origin is a trifurcation.
    By a mass transport argument, infinitely many such trifurcations exist, and thus the cluster at the origin thus has infinitely many ends with positive probability.
    However, almost every component of a pmp hyperfinite Borel graph has at most two ends (see \cite{adams1991amenability} or \cite[Theorem 2.13]{AnushRobin} or \cite[Corollaire IV.24]{gaboriau2000cout}), so this contradicts that $\mathcal{G}_{p_c}$ is hyperfinite.
\end{proof}

\begin{rmk}
    We point out this proof makes transparent why the assumption of nonamenability of the ambient graph $G$ is necessary for this proof to go through. 
    For Bernoulli percolation on an amenable quasi-transitive graph $G$, it is still true that $G[p_c]$ is hyperfinite; however, so is $G[p]$ for all $p \in [0,1]$, since $\mathcal{R}^{\text{full}}$ is hyperfinite, and subequivalence relations of hyperfinite equivalence relations are hyperfinite \cite[Proposition 1.3i)]{JKLcber}. 
    As a result, any proof strategy which attempts to rule out infinite clusters at criticality using hyperfiniteness or lack thereof is doomed to fail.
    In general, ruling out percolation at criticality for amenable graphs appears to be an extremely subtle question. Finally, we point the reader to \cite[Theorem 8.11]{URG} for a generalization of \cref{NoInfiniteClustersAtCriticality} in the setting of {unimodular random graphs}.
\end{rmk}

We highlight that if $G$ is a \textit{nonunimodular} quasi-transitive graph, then a different proof shows that there are no infinite clusters at $p_c(G)$, by results of Tim\'ar and Hutchcroft. This is \cite[Theorem 1]{hutchcroft2016criticalpercolationquasitransitivegraph}; we reprise his case analysis below.

\begin{thm}\label{NonunimodularCase}
    If $G$ is a nonunimodular quasi-transitive graph, then there are no infinite clusters at $p_c(G)$.
\end{thm}

\begin{proof}
    By \cite[Corollary 5.7]{Timar06nonu}, for any nonunimodular quasi-transitive graph $G$, $G[p_c]$ does not have infinitely many infinite clusters.
    By \cite{hutchcroft2016criticalpercolationquasitransitivegraph}, for any quasi-transitive graph $G$ of exponential growth, $G[p_c]$ does not have a unique infinite cluster.
    Since any nonunimodular quasi-transitive graph is nonamenable by Soardi and Woess \cite{Soardi-Woess} and thus has exponential growth, we conclude $G[p_c]$ has no infinite clusters.
\end{proof}

\begin{remark}
    In fact, \textit{a posteriori}, later work of Hutchcroft \cite{Hutchcroft20} shows that for any nonunimodular quasi-transitive graph, there exists a nontrivial phase with infinitely many light clusters; in other words, for such graphs, $p_c(G) < p_u(G)$. If one knows that $p_c(G) < p_u(G)$, then the result of \cite{Timar06nonu} is enough to conclude \cref{NonunimodularCase}. However, the proof of \cite{Hutchcroft20} which shows that $p_c(G) < p_u(G)$ itself uses that there are no infinite clusters at $p_c(G)$, i.e. it relies on \cref{NonunimodularCase}.
\end{remark}

\begin{cor}
    If $G$ is any nonamenable quasi-transitive graph, then there are no infinite clusters at $p_c(G)$.
\end{cor}

We now make some meta-mathematical observations about the two proofs of \cref{NoInfiniteClustersAtCriticality} above.
After unpacking the classical theory of countable Borel equivalence relations above, the reader may notice that the proofs above are not so different from those which appear in \cite{BLPSpc}, \cite{BLPS99inv}.
For instance, in \cite{BLPSpc}, the authors exactly rule out the case of infinitely many infinite clusters by showing that this implies the cluster at the origin has at least three ends with positive probability; they then construct an invariant spanning forest on trifurcations by hand, observe it has infinitely many ends, and use mass transport to obtain a contradiction. 

There are genuine conceptual advantages to using the language of countable Borel equivalence relations: for instance, it allows one to take advantage of the existing literature of measured group theory and to understand heredity properties of invariants like hyperfiniteness, measured Property (T), and treeability.
It is also useful for extending classical proof methods like the mass-transport technique to invariant percolations on nonunimodular graphs, as one can use the theory of \textit{measure-class-preserving} (rather than \textit{probability-measure-preserving}) equivalence relations; see \cite{wamen}, \cite{bell2025heavyrepulsionclustersbernoulli}, \cite{CTT22}, \cite{AnushRobin} for examples.
However, one could argue that the measured-group-theoretic proofs are not ultimately ``shorter", in that the technical content is ultimately still hidden somewhere (for instance, in \cite{JKLcber}).

\subsection{No unique infinite cluster at $p_u$ via nonapproximability}

A well-known conjecture of Benjamini posits that for any quasi-transitive graph $G$ with $p_c(G) < 1$, there is no percolation at $p_c(G)$. 
What happens at $p_u(G)$? 
For amenable graphs, $p_u(G) = p_c(G)$, so this case falls under the conjecture above, but for nonamenable graphs, it is widely believed that $p_c(G)<p_u(G)$ (see \cite[Conjecture 6]{BSbeyond}).
There are examples of nonamenable graphs with a unique infinite cluster at $p_u$ and examples with infinitely many infinite clusters at $p_u$; see the introduction of \cite{HutchcroftPan24rel} for a survey of known results.
The question of uniqueness at $p_u$ appears to be extremely subtle. It is not known, for instance, whether the number of infinite clusters at $p_u(G)$ in a Cayley graph of a fixed group $\Gamma$ depends on the choice of generating set. There is not even a conjectural characterization of the nonamenable groups which have uniqueness at $p_u$, although in \cite{HutchcroftPan24rel}, Pan and Hutchcroft point out that in all known examples, uniqueness at $p_u$ coincides with positivity of the first $L^2$-Betti number.

The following result first appears in the literature in \cite{LyonsSchramm99}, using a modification suggested by Yuval Peres of a weaker argument. 

\begin{thm}\cite[Corollary 6.6]{LyonsSchramm99}\label{NonuniquenessForKazhdanGroups}
    If $G$ is a Cayley graph of an infinite Property (T) group, then there is not a unique infinite cluster at $p_u(G)$.
\end{thm}

In 1999, Schonmann \cite{schonmann1999percolation} showed that for percolation on $T \times \bZ$ for $d$-regular trees $T$ with $d \geq 3$, there are no infinite clusters at $p_u$. Peres extended this to the case of products $X \times Y$ of quasi-transitive graphs $X,Y$ such that $\operatorname{Aut}(X)$ is nonamenable.

\begin{thm}\cite{Peres00}\label{Peres}
    If $X,Y$ are infinite quasi-transitive graphs such that $\operatorname{Aut}(X)$ is nonamenable, then there is not a unique infinite cluster at $p_u(X \times Y)$.
\end{thm}

In a 2016 paper, Gaboriau and Tucker-Drob \cite{GABORIAU20161114} illuminated that both \cref{NonuniquenessForKazhdanGroups} and \cref{Peres} are true for the ``same" reasons. Indeed, they observed a general obstruction to uniqueness at $p_u$ by way of the nonapproximability of the associated cluster graphings.
We remark that a comprehensive theory of obstructions to uniqueness via subgroup relativization was developed by Hutchcroft and Pan in \cite{HutchcroftPan24rel}, who showed that amenable normal (more generally, $wq$-normal) subgroups of exponential growth are also obstructions to uniqueness.

We begin by reprising some definitions from \cite{GABORIAU20161114}. 
Below, we consider countable Borel equivalence relations $\mathcal{R}$ on a standard probability space $(X, \mu)$, where $\mathcal{R}$ is probability-measure-preserving (pmp).

\begin{defn}
    An approximation $(\mathcal{R}_n)_{n \geq 0}$ to a countable Borel equivalence relation $\mathcal{R}$ on $(X,\mu)$ is an increasing exhausting sequence of countable Borel equivalence relations with $\bigcup_{n}\mathcal{R}_{n}= \mathcal{R}$. We say that an approximation is trivial if there exists some $n \in \mathbb{N}$ and some $A \subseteq X$ of $\mu$-positive measure such that $\mathcal{R}_n|_{A}= \mathcal{R}|_A$. 
\end{defn}

\begin{defn}
    We say a countable Borel equivalence relation $\mathcal{R}$ on $(X,\mu)$ is nonapproximable if all approximations to $\mathcal{R}$ are trivial.
\end{defn}

In the usual way, the definitions above extend to actions of a countable group: say that an action $\alpha: \Gamma \curvearrowright (X, \mu)$ is approximable if the resulting orbit equivalence relation $\mathcal{R}_\Gamma$ is approximable.

\begin{prop}
    The following equivalence relations are nonapproximable:
\begin{enumerate}
    \item \cite[Proposition 16]{pichot2007theorie} Every pmp action of a Kazhdan Property (T) group is nonapproximable.
    \item \cite[Theorem 1.1]{GABORIAU20161114} Let $\Gamma$ be a group generated by two infinite commuting finitely generated subgroups $H$ and $K$. For any free pmp action of $\Gamma$ on a standard probability space $(X,\mu)$ such that $H$ acts strongly ergodically and $K$ acts ergodically, the orbit equivalence relation $\mathcal{R}_\Gamma$ is nonapproximable.
    \begin{enumerate}
        \item In particular if $\Gamma = H \times K$ is nonamenable, then one of $H,K$ must be nonamenable; after renaming, assume $K$ is nonamenable. Then the Bernoulli action of $\Gamma$ on $2^\Gamma$ will satisfy the properties above, and thus is a free ergodic pmp nonapproximable action of $\Gamma$.
    \end{enumerate}
\end{enumerate}
\end{prop}

Recall the Bernoulli action of $\Gamma$ on $2^\Gamma$ is defined as follows.

\begin{defn}[Bernoulli action]
    Given a group $\Gamma$, we define the Bernoulli action (or shift) $\Gamma \acts 2^\Gamma$ as follows. Equip $X = 2^\Gamma$ with probability measure $\mu$, where $\mu$ is the product measure of Bernoulli$(1/2)$ measures on each marginal. 
    Let $\Gamma$ act on a sequence $(x_{\tau})_{\tau \in \Gamma}$ by $(\gamma \cdot x)_{\tau} = x_{\gamma^{-1} \tau}$. 
\end{defn}

\begin{remark}
    The Bernoulli action of $\Gamma$ is essentially free, meaning after throwing away a null set, we obtain a free action of $\Gamma$ on $(X,\mu)$. It is also ergodic and pmp.
\end{remark}

The following theorem is implicit in \cite{GABORIAU20161114} and used to deduce \cite[Theorem 2.1]{GABORIAU20161114} from \cite[Theorem 1.1]{GABORIAU20161114}; we include a proof for the convenience of the reader, and to highlight the use of the cluster graphing construction.

\begin{thm}
    Let $\Gamma$ be a countable group. If there exists a free ergodic pmp nonapproximable action of $\Gamma$ on a standard probability space $(X,\mu)$, then for any Cayley graph $G$ of $\Gamma$, there is not a unique infinite cluster at $p_u(G)$.
\end{thm}

\begin{proof}
    Let $G$ be a Cayley graph of $\Gamma$ and $\alpha: \Gamma \acts X$ a free ergodic pmp nonapproximable action on $(X,\mu)$. Denote by $E_\Gamma$ the resulting orbit equivalence relation on $X$.
    Using this free action, we construct the family of cluster equivalence relations $(\mathcal{R}_p)_{p \in [0,1]}$ on a common probability space $(Y,\nu)$ corresponding to $G[p]$ as at the end of Section 2.C in the \nameref{sec: Appendix}.
    
    Notice that we can identify $\mathcal{R}_1 = \mathcal{R}^{\text{full}}$ with the restriction of $E_\Gamma$ to a complete section $X_o$ of $X \times \{o\}$, see \cite[5.B.1]{CTT22}. 
    Since $E_\Gamma$ is ergodic and nonapproximable, the restriction of this nonapproximable equivalence relation to any positive-measure set is also nonapproximable. Thus, $\mathcal{R}^{\text{full}}$ on $(Y,\nu)$ is also nonapproximable. Moreover, by the same trick, for any positive-measure set $B \subseteq Y$, the restriction $\mathcal{R}^{\text{full}}|_B$ is also nonapproximable.

    Suppose towards contradiction there is a unique infinite cluster at $p_u(G)$. 
    Fix a distinguished vertex $o \in V(G)$, e.g. the identity of $\Gamma$.
    Let $A = \{x \in X: o \leftrightarrow \infty \text{ in } \pi_{p_u}(x)\}$ be the set of configurations in which $o$ connects to that unique infinite cluster at intensity $p_u$. 
    
    Now consider 
    $B = \{[x,u]_{\Gamma} \in (X \times V)/\Gamma: u \leftrightarrow \infty \text{ in } \pi_{p_u}(x)\}$, the pushforward of $A$ under the surjection $x \mapsto [x,o]_{\Gamma}$. Equivalently, $B =\{[x,u]_{\Gamma} \in (X \times V)/\Gamma: \text{the } \mathcal{R}_{p_u}\text{-component of $[x,u]_\Gamma$ is infinite}\}$.
    The fact there is a unique infinite cluster at $p_u(G)$ implies that the induced connectedness relations coincide on $B$, i.e. $\mathcal{R}^{\text{full}}|_{B} = \mathcal{R}_{p_u}|_{B}$. (Indeed, $\mathcal{R}_{p_u} \subseteq \mathcal{R}^\text{full}$ always, and if $[x,u]_\Gamma, [x, v]_{\Gamma} \in B$ with $[x,u]_\Gamma \mathcal{R}^{\text{full}}[x, v]_{\Gamma}$, then $u \leftrightarrow v$ in $\pi_{p_u}(x)$ since they both connect to the same unique infinite cluster, hence $[x,u]_\Gamma \mathcal{R}_{p_u} [x,v]_\Gamma$.)
    
    Now let $(p_n)_{n \geq 0}$ be a sequence converging to $p_u$ from below.
    This approximation must be nontrivial: if $\mathcal{R}_{p_n}$ agrees with $\mathcal{R}^{\text{full}}$ on a positive-measure set for some $p_n < p_u$ then $G[p_n]$ has a unique infinite cluster with positive probability, contradicting the definition of $p_u$.
    However, this contradicts that $\mathcal{R}^{\text{full}}|_B = \mathcal{R}_{p_u}|_B$ is nonapproximable.
\end{proof}

\subsection{Other applications}

We now briefly and non-exhaustively mention other appearances of the cluster graphing construction in the percolation theory literature. Martineau studied ergodicity properties of cluster graphings extensively in \cite{martineau2016ergodicity}. 
Tim\'ar and Terlov \cite{wamen} used cluster graphings to study $\w$-nonamenable graphs and (among other beautiful results) show that there are no heavy clusters at $p_h$ for such graphs, using similar arguments to the proofs of \cref{NoInfiniteClustersAtCriticality} above. 
Chen, Terlov, and Tserunyan use cluster graphings \cite{CTT22} to study a generalization of the Free Minimal Spanning Forest (FMSF) which is introduced in the same paper.
Using the cluster graphing construction, the authors of \cite{bell2025heavyrepulsionclustersbernoulli} prove that for Bernoulli$(p)$ percolation on a nonunimodular quasi-transitive graph, any two heavy clusters neighbour in a $\w$-finite set. 
Finally, we remark that cluster graphings have a continuum analogue in Palm equivalence relations, which are devices for understanding invariant point processes on locally compact second countable unimodular groups and their associated symmetric spaces; see \cite{grabowski2025unimodularrandomgraphsproperty}, \cite{fraczyk2023poisson}, \cite{abert2022point}.

\section{Appendix}\label{sec: Appendix}

\subsection{Notation}

If $G=(V,E)$ is a locally finite graph, we let $G[p]$ denote the random graph formed by independently at random deleting each edge with probability $1-p$ and retaining each edge with probability $p$. 
This is a random variable valued in $2^E$, and we denote its law by $\pr_p
$.

\subsection{Countable Borel equivalence relations}

We briefly recall some standard definitions from the theory of countable Borel equivalence relations (CBERs). We refer the reader to \cite{kechrismiller} for more detail. 

\begin{definition}
    Given a standard Borel space $X$ and an equivalence relation $\mathcal{R} \subseteq X^2$, we say that $\mathcal{R}$ is a {countable Borel equivalence relation} on $X$ if $\mathcal{R}$ is a Borel subset of $X^2$ and each $\mathcal{R}$-equivalence class is countable. If each $\mathcal{R}$-equivalence class is finite, we call $\mathcal{R}$ a finite equivalence relation.
\end{definition}

In particular, the Feldman-Moore theorem says that every countable Borel equivalence relation is the orbit equivalence relation of a Borel action by a countable group.

\begin{thm}[Feldman-Moore]
    Let $\mathcal{R}$ be a countable Borel equivalence relation on a standard Borel space $X$. Then there is a Borel action $\alpha: \Gamma \acts X$ of a countable discrete group $\Gamma$ such that $\mathcal{R} = \mathcal{R}_\Gamma$, where $(x, y) \in \mathcal{R}_\Gamma$ if and only if $x,y$ lie in the same $\Gamma$-orbit. 
\end{thm}

In light of the theorem above, the two definitions below make sense.
\begin{defn}[Probability-measure-preserving]
    A CBER $\mathcal{R} = \mathcal{R}_\Gamma$ on a standard probability space $(X, \mu)$ is probability-measure-preserving (or pmp) if the action of $\Gamma \acts (X,\mu)$ is probability-measure-preserving, meaning $\gamma_\ast \mu = \mu$ for all $\gamma \in \Gamma$. 
\end{defn}

\begin{defn}[Measure-class-preserving]
    A CBER $\mathcal{R} = \mathcal{R}_\Gamma$ on a standard probability space $(X, \mu)$ is measure-class-preserving (or mcp) if the action of $\Gamma \acts (X,\mu)$ is measure-class-preserving, meaning $\gamma_\ast \mu$ is absolutely continuous with respect to $\mu$ and $\mu$ is absolutely continuous with respect to $\gamma_\ast \mu$  for all $\gamma \in \Gamma$. 
\end{defn}

\begin{defn}[Hyperfinite]
    A CBER $\mathcal{R}$ on a standard probability space $(X, \mu)$ is (Borel) hyperfinite if there exists an increasing sequence $(\mathcal{R}_n)_{n \geq 0}$ of finite Borel equivalence relations such that $\mathcal{R} = \bigcup_n \mathcal{R}_n$.
\end{defn}

\begin{definition}
    A CBER $\mathcal{R}$ on a standard probability space $(X,\mu)$ is $\mu$-hyperfinite if there is a $\mu$-conull Borel set $A$ such that $\mathcal{R}|_A$ is hyperfinite.
\end{definition}

Since we work in a measured context, we abuse notation and use the word $\mu$-hyperfinite interchangeably with the word hyperfinite throughout this note.

\begin{definition}
    A Borel graph $\mathcal{G}$ on a standard Borel space $X$ is a graph with vertex set $X$ such that the connectedness relation $\mathcal{R}_{\mathcal{G}}$ is a countable Borel equivalence relation.
\end{definition}

We can extend the definitions of hyperfiniteness, pmp, etc. to Borel graphs in the usual way. 
For example, we say that a Borel graph $\mathcal{G}$ is hyperfinite if its connectedness relation $\mathcal{R}_{\mathcal{G}}$ is hyperfinite; similarly, we say that a Borel graph $\mathcal{G}$ on $(X,\mu)$ is pmp if its connectedness relation $\mathcal{R}_{\mathcal{G}}$ is pmp, and so on.

\begin{definition}
    Let $\mathcal{R}$ be a countable Borel equivalence relation on a standard probability space $(X,\mu)$. A \textit{graphing} of $\mathcal{R}$ is a Borel graph $\mathcal{G}$ whose connectedness relation $\mathcal{R}_{\mathcal{G}}$ is exactly $\mathcal{R}$.
\end{definition}

\subsection{The cluster graphing}

The following construction is due to Gaboriau. We refer the reader to \cite{Gaboriau05} for more details.

Let $G = (V,E)$ be a connected, locally finite, transitive, countable graph, $o \in V(G)$ a distinguished vertex, and $\Gamma = \Aut(G)$. 
For any invariant percolation, that is a $\Gamma$-invariant measure $\vP$ on $2^{E}$, we construct a Borel graph $\mathcal{G}^{\text{cl}}$ as follows.
Fix a free pmp action of $\Gamma$ on a standard probability space $(X, \mu)$, and let $\pi: (X, \mu) \rightarrow (2^{E}, \mathbf{P})$ be a $\Gamma$-equivariant factor map (in particular, this means $\pi_{\ast} \mu = \mathbf{P}$).
Think of $\pi$ as the random variable valued in $2^{E}$ with law $\vP$.

Consider the Cartesian product $X \times V$, and define the {product graph} $\mathcal{G}$ on $X \times V$ as follows: given $x,y \in X$ and $u,v \in V$, let $(x,u)$ and $(y,v)$ be adjacent in $\mathcal{G}$ if and only if $x = y$ and $(u,v) \in E$. 
Notice that $\Gamma$ acts diagonally on $X \times V$ (and thus $\mathcal{G}$).
Moreover, $\mathcal{G}$ is $\Gamma$-invariant: for all $\gamma \in \Gamma$, we have that $((x,u),(y,v)) \in E(\mathcal{G})$ if and only if $((\gamma x, \gamma u), (\gamma y, \gamma v)) \in E(\mathcal{G})$. 
Let $\mathcal{R}_{\mathcal{G}}$ denote the connectedness relation of $\mathcal{G}$.

Consider now the {quotient graph} $\mathcal{G}/\Gamma$ on the vertex set $Y:=(X \times V)/\Gamma$.
Notice that $Y$ is a standard Borel space since it can be naturally identified with a Borel transversal $X_o \subseteq X \times \{o\}$ of the action of $\Gamma_o$ (the stabilizer of $o$) on $X \times \{o\}$ by the Becker-Kechris theorem, see \cite{CTT22}.
To define a measure on $Y$, we define a function $f: X \rightarrow Y$ by $f(x) = [(x,o)]_\Gamma$. 
Then $f$ is surjective, so $\nu = f_\ast \mu$ is a probability measure on $Y$. 

We can describe the quotient graph $\mathcal{G} / \Gamma$ on $Y$ explicitly as follows: two representatives $[(x,u)]_{\Gamma}$ and $[(y,v)]_{\Gamma}$ are adjacent in $\mathcal{G}/\Gamma$ if and only if there exists $\gamma \in \Gamma$ so that $x = \gamma y$ and $(u, \gamma v) \in E(G)$. 
In other words, one can change the representative of $[(y,v)]_{\Gamma}$ to a representative $[(x, \gamma v)]_{\Gamma}$ where the first coordinates agree and the second coordinates form an edge of $E$. 
One can check that each connected component of $\mathcal{G}/ \Gamma$ is isomorphic to $G$. 
The connectedness equivalence relation $\mathcal{R}^{\text{full}} = \mathcal{R}_{\mathcal{G}/\Gamma}$ of $\mathcal{G}/\Gamma$ is called the full equivalence relation by Gaboriau in \cite[Section 1.1]{Gaboriau05} and denoted by $\mathcal{E}_V$ in \cite{CTT22}.

\begin{defn}
    The cluster graphing $\mathcal{G}^{\mathrm{cl}} \subseteq \mathcal{G}/\Gamma$ is the graph on $Y$ defined by placing edges between $[(x,u)]_\Gamma$ and $[(x,v)]_\Gamma$ whenever $(u,v) \in \pi(x)$. 
\end{defn}
We let $\mathcal{R}^{\text{cl}}$ denote the connectedness relation of the Borel graph $\mathcal{G}^{\mathrm{cl}}$. 
We summarize some basic properties of this graphing below.

\begin{prop}
    The following holds:
    \begin{enumerate}
        \item The $\mathcal{G}^{\mathrm{cl}}$-component of $\pi({x}) = [(x,u)]_{\Gamma} \in Y$ is isomorphic to the cluster of $u$ in $\pi(x)$.
        \item 
        If $G$ is unimodular, then $\mathcal{G}^{\text{cl}}$ is pmp by \cite[Theorem 2.1]{Gaboriau05}, \cite[Lemma 5.7]{CTT22}. 
        \begin{enumerate}
            \item More generally, if $G$ is a quasi-transitive graph, the equivalence relation $\mathcal{R}_{\mathcal{G}^{\mathrm{cl}}}$ is measure-class-preserving with respect to $\nu$, and the corresponding Radon-Nikodym cocycle is given by
\begin{equation*}\label{cocycle_cber}
\Tilde{\w}^{[x,u]_\Gamma}([x,v]_\Gamma) = \w_\Gamma^u (v).
    \end{equation*}
    In other words, it is essentially induced by $\w_\Gamma$.
        \end{enumerate}
    \end{enumerate}
\end{prop}

In the special case where $\vP = \pr_p$, we let $\mathcal{R}_p$ denote the corresponding cluster equivalence relation.
Using the standard Uniform $[0,1]$ coupling of Bernoulli($p$) percolation, we may define a continuum family of countable Borel equivalence relations $(\mathcal{R}_p)_{p \in [0,1]}$ on the same probability space $Y = (X \times V)/\Gamma$ such that whenever $p<q \in [0,1]$, one has $\mathcal{R}_p \subseteq \mathcal{R}_q$. We let $(X,\mu)$ denote the common probability space and $\pi_p: (X, \mu) \rightarrow (2^E, \pr_p)$ denote the $\Gamma$-equivariant factor map which realizes $G[p]$. Notice that the action of $\Gamma$ on the probability space $[0,1]^E$ is already essentially free, as observed in \cite[Sections 8.4-5]{gaboriau2002arbres}. We refer the reader to \cite[Sections 8.4-5]{gaboriau2002arbres} and \cite[1.3.d]{Gaboriau05} for further detail.

\subsection{Acknowledgements}

I thank Damien Gaboriau for numerous helpful comments which greatly improved the stylistic and mathematical content of this note, and for pointing out an error in a previous version. 
I also thank Robin Tucker-Drob for the second proof of \cref{NoInfiniteClustersAtCriticality} and for helping me understand a key step of \cite{GABORIAU20161114}. Finally I acknowledge my advisor Anush Tserunyan (who first introduced me to the cluster graphing) for her constant support and insightful conversations, mathematical and otherwise. 

\bibliographystyle{alpha}
\bibliography{bibliography} 

\end{document}